\documentclass{amsart}
\usepackage{mathtools, microtype, mathrsfs, xfrac, hyperref, amssymb, enumerate, xspace, stmaryrd}
\usepackage[utf8]{inputenc}
\usepackage{tikz-cd}
\usepackage{pdflscape}
\usepackage{bm}
\usepackage{url}

\usepackage{todonotes}

\numberwithin{equation}{section}

\numberwithin{subsection}{section}

\allowdisplaybreaks[1]

\newtheorem*{namedtheorem}{\theoremname}
\newcommand{\theoremname}{testing}

\newtheorem{theorem}{Theorem}

\newtheorem{proposition-definition}[theorem]
{Proposition-Definition}

\newtheorem{lemma}[theorem]{Lemma}
\newtheorem*{theorem*}{Theorem}

\theoremstyle{definition}

\newtheorem*{question*}{Question}

\theoremstyle{remark}

\renewcommand{\mathcal}{\mathscr}

\newcommand\cC{\mathcal{C}} 
 
\newcommand\cG{\mathcal{G}}

 \newcommand\cN{\mathcal{N}}
 \newcommand\cP{\mathcal{P}}

\newcommand\CC{\mathbb{C}}

 \newcommand\PP{\mathbb{P}}
 \newcommand\RR{\mathbb{R}}

 \newcommand\ZZ{\mathbb{Z}}

\newcommand\bC{\mathbf{C}}

 \newcommand\bP{\mathbf{P}}

\newcommand\arr{\ifinner\to\else\longrightarrow\fi}

\newcommand\arrto{\ifinner\mapsto\else\longmapsto\fi}

\def\displaytimes_#1{\mathrel{\mathop{\times}\limits_{#1}}}

\def\displayotimes_#1{\mathrel{\mathop{\bigotimes}\limits_{#1}}}

\newcommand\aut{\operatorname{Aut}}

\newcommand\spec{\operatorname{Spec}}

\newcommand{\underaut}{\mathop{\underline{\mathrm{Aut}}}\nolimits}

\newlength{\ignora}

\renewcommand{\setminus}{\smallsetminus}

\newcommand{\GL}{\mathrm{GL}}

\newcommand{\PGL}{\mathrm{PGL}}

\newcommand{\gal}{\operatorname{Gal}}

\DeclareFontFamily{U}{mathx}{\hyphenchar\font45}
\DeclareFontShape{U}{mathx}{m}{n}{
      <5> <6> <7> <8> <9> <10>
      <10.95> <12> <14.4> <17.28> <20.74> <24.88>
      mathx10
      }{}
\DeclareSymbolFont{mathx}{U}{mathx}{m}{n}
\DeclareFontSubstitution{U}{mathx}{m}{n}
\DeclareMathAccent{\widecheck}{0}{mathx}{"71}
\DeclareMathAccent{\wideparen}{0}{mathx}{"75}

\renewcommand{\epsilon}{\varepsilon}

\newcommand{\cha}{\operatorname{char}}

\begin{document}

\title{The field of moduli of sets of points in $\PP^{2}$}

\author{Giulio Bresciani}

\begin{abstract}
	For every $n\ge 6$, we give an example of a finite subset of $\mathbb{P}^{2}$ of degree $n$ which does not descend to any Brauer-Severi surface over the field of moduli. Conversely, for every $n\le 5$ we prove that a finite subset of degree $n$ always descends to a $0$-cycle on $\mathbb{P}^{2}$ over the field of moduli.
\end{abstract}

\address{Scuola Normale Superiore\\Piazza dei Cavalieri 7\\
56126 Pisa\\ Italy}
\email{giulio.bresciani@gmail.com}


\maketitle

Let $k$ be a field with separable closure $K$, and $S\subset\PP^{2}(K)$ a finite subset of order $n$. The \emph{field of moduli} $k_{S}$ of $S$ is the subfield of $K$ of elements fixed by Galois automorphisms $\sigma\in\gal(K/k)$ such that $\sigma(S)$ is linearly equivalent to $S$. We study the problem of whether $S$ descends to a $0$-cycle on $\PP^{2}_{k(S)}$, or more generally on a Brauer-Severi surface over $k_{S}$.

A. Marinatto \cite{marinatto} studied the analogous problem over $\PP^{1}$. He showed that, if $n$ is odd or equal to $4$, then $S$ descends to a divisor over $\PP^{1}_{k_{S}}$. Furthermore, he has given counterexamples where $S$ does not descend to $\PP^{1}_{k_{S}}$ for every $n\ge 6$ even. All of his counterexamples descend to a Brauer-Severi curve, though. In \cite{giulio-divisor}, we have shown that if $n=6$ then $S$ always descend to some Brauer-Severi curve, while there are counterexamples for every $n\ge 8$ even. 

Fields of moduli of curves, possibly with marked points, received a lot of attention, see \cite{debes-douai} \cite{debes-emsalem} \cite{huggins} \cite{cardona-quer} \cite{kontogeorgis} \cite{hidalgo}. Furthermore, there are results about abelian varieties, most famously Shimura's result that a generic, principally polarized, odd dimensional abelian variety is defined over the field of moduli \cite{shimura}, and about fields of moduli of curves in $\PP^{2}$ \cite{roe-xarles} \cite{badr-bars-garcia} \cite{badr-bars} \cite{artebani-quispe} \cite{giulio-p2}. Here is our result.

\begin{theorem}\label{theorem:sets}
	Assume $\cha k\neq 2$. Let $S\subset\PP^{2}(K)$ be a finite set of $n$ points with field of moduli $k_{S}$. If $n\le 5$, then $S$ descends to a finite subscheme of $\PP^{2}_{k_{S}}$. For every $n\ge 6$, there exists a subset $S\subset\PP^{2}(\CC)$ with field of moduli equal to $\RR$ which does not descend to $\PP^{2}_{\RR}$. 
\end{theorem}

Notice that $\PP^{2}_{\RR}$ is the only Brauer-Severi surface over $\RR$, hence our counterexamples do not descend to any Brauer-Severi surface over $\RR$.

\section{Notation and conventions}

Let $Z\subset\PP^{2}$ be a closed subscheme, and $g\in\PGL_{3}(K)$ a projective linear map. We say that $g$ \emph{stabilizes} $Z$, or that $Z$ is $g$-invariant, if $g(Z)=Z$. We say that $g$ \emph{fixes} $Z$ if $g$ restricts to the identity on $Z$. If $G\subset\PGL_{3}(K)$ is a finite subgroup, we say that $G$ stabilizes (resp. fixes) $Z$ if every element $g\in G$ stabilizes (resp. fixes) $Z$. The \emph{fixed locus} of $g$ (resp. $G$) is the subspace of points $x\in X$ with $gx=x$ (resp. $\forall g\in G:gx=x$).

Let $S\subset\PP^{2}(K)$ be a finite subset with finite automorphism group. Up to replacing $k$ with $k_{S}$, we may assume that $k$ is the field of moduli.

A \emph{twisted form} of $(\PP^{2}_{K},S)$ over a $k$-scheme $M$ is the datum of a projective bundle $P\to M$ and a closed subscheme $Z\subset P$ such that $(P_{K},Z_{K})$ is étale locally isomorphic to $(\PP^{2}_{K},S)\times M_{K}$ (notice that if we do not assume that $k$ is the field of moduli this definition is not correct, e.g. $(\PP^{2}_{K},S)$ would not define a twisted form of $(\PP^{2}_{K},S)$).

The fibered category $\cG_{S}$ of twisted forms of $S$ is a finite gerbe over $\spec k$ called the \emph{residual gerbe} of $S$, see \cite{giulio-angelo-moduli}. By Yoneda's lemma we have a universal projective bundle $\cP_{S}\to\cG_{S}$, its base change to $K$ is the quotient stack $[\PP^{2}/\underaut_{K}(S)]\to[\spec K/\underaut_{K}(S)]$.

Another way of constructing $\cP_{S}\to\cG_{S}$ is the following. Let $\cN_{S}\subset\aut_{k}(\PP^{2}_{K})$ be the subgroup of $k$-linear automorphisms $\tau$ of $\PP^{2}_{K}$ such that $\tau(S)=S$, the fact that $k$ is the field of moduli implies that $\cN_{S}$ is an extension of $\gal(K/k)$ by $\cN_{S}\cap\aut_{K}(\PP^{2}_{K})=\aut_{K}(\PP^{2},S)$ (see \cite[\S 3]{giulio-fmod} for details). We have an induced action of $\cN_{S}$ on $\spec K$ with the natural projection $\cN_{S}\subset\aut_{k}(\PP^{2}_{K})\to\gal(K/k)$, and the finite étale gerbe $\cG_{S}$ is the quotient stack $[\spec K/\cN_{S}]$: the natural map $\spec K\to\cG_{S}$ associated with the trivial twist of $S$ is a pro-étale, Galois covering with Galois group equal to $\cN_{S}$. Similarly, we can view $\cP_{S}$ as the quotient stack $[\PP^{2}_{K}/\cN_{S}]$.

Twisted forms of $S$ over Brauer-Severi surfaces over $k$ correspond to rational points of $\cG_{S}$. If $|S|$ is prime with $3$ and $S$ descends to a $0$-cycle over some Brauer-Severi surface, then clearly this Brauer-Severi surface is trivial since its index divides both $3$ and $|S|$.

Denote by $\bP_{S}$ the coarse moduli space of $\cP_{S}$, i.e. $\PP^{2}_{K}/\cN_{S}$, since the action of $\aut_{K}(S)$ on $\PP^{2}_{K}$ is faithful then the natural map $\cP_{S}\to\bP_{S}$ has a birational inverse $\bP_{S}\dashrightarrow\cP_{S}$ which, by composition, gives us a rational map $\bP_{S}\dashrightarrow\cG_{S}$.

\section{Case $n\le 5$}

The cases $n\le 3$ are trivial. If $n=4$, the only non-trivial case is the one in which the four points are all contained in a line, and it follows from \cite[Proposition 13]{giulio-divisor}.

Assume $n=5$. Let $S\subset \PP^{2}(K)$ be a finite subset of degree $5$ with field of moduli $k$, since $5$ is prime with $3$ it is enough to show that $\cG_{S}(k)\neq\emptyset$. We split the analysis in three cases: either $S$ contains $4$ points in general position, or it is contained in the union of two lines each containing at least three points of $S$, or it is contained in the union of a line and a point.

\subsection{$S$ contains $4$ points in general position}
Since we are assuming that there are $4$ points of $S$ in general position, there are two possibilities: either all $5$ points are in general position i.e. there is no line containing $3$ of them, or there is a unique line containing exactly $3$ points of $S$. Denote by $C$ the unique non-degenerate conic passing through all the points of $S$ in the first case, while in the second case $C$ is the unique line containing $3$ points of $S$. 

In any case, $C$ is a rational curve uniquely determined by $S$. Because of this, $\cN_{S}$ stabilizes $C$, consider the quotient $\cC=[C/\cN_{S}]\subset\cP_{S}$ and let $\bC=C/\cN_{S}$ be the coarse moduli space of $\cC$. Notice that, since $C\cap S\ge 3$, the subgroup of $\aut_{K}(\PP^{2},S)$ fixing $C$ has at most $2$ elements, hence the map $\cC\to\bC$ is either birational or generically a gerbe of degree $2$. In any case, since $\cha k\neq 2$ by the Lang-Nishimura theorem for tame stacks \cite[Theorem 4.1]{giulio-angelo-valuative} applied to a birational inverse $\bP_{S}\dashrightarrow\cP_{S}$ we get a generic section $\bC\dashrightarrow \cC$.

Since $C\cap S$ has odd degree, there exists an odd $d$ such that $C\cap S$ contains an odd number of orbits of order $d$, let $O\subset C\cap S$ be their union. Clearly, $O$ is stabilized by $\cN_{S}$, hence $O/\aut(\PP^{2},S)\subset C/\aut(\PP^{2},S)$ descends to a divisor of odd degree of $\bC$. In particular, we get that $\bC\simeq\PP^{1}_{k}$. Since we have a map $\bC\dashrightarrow \cC\to\cP_{S}\to\cG_{S}$, this implies that $\cG_{S}(k)\neq\emptyset$ if $k$ is infinite. If $k$ is finite, the statement follows from the fact that $\cN_{S}\to\gal(K/k)\simeq\hat{\ZZ}$ is split and hence $\cG_{S}(k)\neq\emptyset$.

\subsection{$S$ contained in the union of two lines}

Assume that $S$ is contained in the union of two lines $L,L'$ each containing at least $3$ points. Up to changing coordinates, we may assume that $(0:0:1),(0:1:0),(1:0:0)\in S$. It is now clear that, up to permuting the coordinates and multiplying them by scalars, we might assume that
\[S=\{(0:0:1),(0:1:0),(1:0:0),(0:1:1),(1:0:1)\},\] 
which is clearly defined over $k$.

\subsection{$S$ contained in the union of a line and a point}

Suppose that $S$ is contained in the union of a line $L$ and a point $p$, choose coordinates such that $p=(0:0:1)$ and $L=\PP^{1}$ is the line $\{(s:t:0)\}$.

The field of moduli of $(\PP^{2}_{K},S)$ is equal to the field of moduli of $(\PP^{1}_{K},S\cap \PP^{1}_{K})$: given $\sigma\in\gal(K/k)$, clearly $\sigma^{*}(\PP^{2}_{K},S)\simeq (\PP^{2}_{K},S)$ if and only if $\sigma^{*}(\PP^{1}_{K},S\cap\PP^{1}_{K})\simeq (\PP^{1}_{K},S\cap\PP^{1}_{K})$. By \cite[Proposition 13]{giulio-divisor}, $\PP^{1}_{K}\cap S$ descends to a closed subset of $\PP^{1}_{k}$. It follows that $S$ descends to a closed subset of $\PP^{2}_{k}$.

\section{Case $n\ge 6$}\label{sect:ge6}

Let us now construct a counterexample with $k=\RR$, $K=\CC$ for every $n\ge 6$.

If $n\ge 6$, then either $n=2m+4$ or $n=2m+5$ for some $m\ge 1$. Given $a_{1},\dots,a_{m}\in \CC^{*}$, $|a_{i}|\neq 1$, define
\[F=\left\{(\pm 1:0:1),(0:\pm 1:1)\right\},\]
\[S=\left\{(a_{i}:1:0),(1:-\bar{a}_{i}:0)\right\}_{i}\cup F,\]
\[S'=S\cup\left\{(0:0:1)\right\},\]
then $|S|=2m+4$, $|S'|=2m+5$. The matrix $\begin{psmallmatrix}  & -1 &   \\ 1 & &	\\ & & 1 \end{psmallmatrix}$ gives a linear equivalence between $S,S'$ and their respective complex conjugates, hence they both have field of moduli equal to $\RR$. Let us show that $S$ is not defined over $\RR$ (the case of $S'$ is analogous).

Let $M\in\PGL_{3}(\CC)$ be the image of $\begin{psmallmatrix} -1 &  &   \\  & -1 &	\\ & & ~1 \end{psmallmatrix}$. Clearly, $M$ is a non-trivial automorphism of both $S$ and $S'$. 

\begin{lemma}\label{lemma:aut}
	For a generic choice of $a_{1},\dots,a_{m}\in \CC$, $|a_{i}|\neq 1$ and $m\ge 1$, $\aut_{\CC}(\PP^{2},S)=\aut_{\CC}(\PP^{2},S')=\left<M\right>$.
\end{lemma}

\begin{proof}
	For a generic choice of $a_{1},\dots,a_{m}$, there are exactly two lines containing exactly three points of $S'$, hence their point of intersection $(0:0:1)$ is fixed by $\aut_{\CC}(\PP^{2},S')\subset\aut_{\CC}(\PP^{2},S)$. Since $M\in\aut(\PP^{2},S')$, it is enough to show $\aut(\PP^{2},S)=\left<M\right>$.
	
	Let $L=\left\{(s:t:0)\right\}$ be the line at infinity. We first show that it is stabilized by $\aut_{\CC}(\PP^{2},S)$ for a generic choice of the $a_{i}$. If $m\ge 2$ this is obvious, since it is the only line containing at least four points of $S$.
	
	Assume $m=1$, $a_{1}=a$. Since the stabilizer of $F$ in $\GL_{2}(\CC)$ is finite and acts $\CC$-linearly on $\PP^{2}$, for a generic $a$ there is no element of $\aut_{\CC}(\PP^{2},S)$ swapping $(a:1:0)$ and $(1:-\bar{a}:0)$. Assume by contradiction that $L$ is not stabilized. We may then also assume that the orbit of $(1:-\bar{a}:0)$ intersects $F$ (if this happens for $(a:1:0)$ but not $(1:-\bar{a}:0)$, we just change coordinates). 
	
	Since $M$ is an element of order $2$ of $\aut_{\CC}(\PP^{2},S)$ acting as a double transposition of $F$ and no element of $\aut_{\CC}(\PP^{2},S)$ swaps $(a:1:0)$ and $(1:-\bar{a}:0)$, it follows that there exists an element $g\in\aut_{\CC}(\PP^{2},S)$ swapping some $p\in F$ and $(1:-\bar{a}:0)$. In particular, $g$ permutes the other four points $F\cup\{(a:1:0)\}\setminus \{p\}$, we may thus think of $g$ as an element of $S_{4}$. Since the four points $F\cup\{(a:1:0)\}\setminus \{p\}$ are in general position, for each element $\sigma\in S_{4}$ there exists $\phi_{\sigma}\in\PGL_{3}(\CC)$ acting as $\sigma$ on $F\cup\{(a:1:0)\}\setminus \{p\}$, and we may write $\phi_{\sigma}$ as a $3\times 3$ matrix whose entries are algebraic functions of $a$. Since complex conjugation is not algebraic, for a generic choice of $a$ we have $\phi_{\sigma}(p)\neq (1:-\bar{a}:0)$ for every $\sigma\in S_{4}$. This implies that for a generic choice of $a$ the automorphism $g$ cannot exist (since $g(p)= (1:-\bar{a}:0)$), and hence $L$ is stabilized.
	
	If $L$ is stabilized, then $F$ is stabilized, too. The point $(0:0:1)$ is the only point of intersection in $\PP^{2}\setminus (L\cup S)$ of two lines passing through two points of $F$, hence it is fixed by $\aut_{\CC}(\PP^{2},S)$. This implies that $\aut_{\CC}(\PP^{2},S)\subset\GL_{2}(\CC)\subset\PGL_{3}(\CC)$.

	The subgroup of $\GL_{2}(\CC)$ stabilizing $F$ is $D_{4}$, hence $\aut_{\CC}(\PP^{2},S)\subset D_{4}$. The center of $D_{4}$ is $\left<-1\right>=\left<M\right>\subset\GL_{2}(\CC)$, which is also the kernel of $D_{4}\to\PGL_{2}(\CC)$. Since $D_{4}/\left<M\right>$ is finite and acts by $\CC$-linear automorphisms on $L$, for a generic choice of $a_{1},\dots,a_{m}$ the intersection $S\cap L$ is not stabilized by any non-trivial element of $D_{4}/\left<M\right>\subset\PGL_{2}(\CC)$. It follows that $\aut_{\CC}(\PP^{2},S)=\left<M\right>$.
\end{proof}

Write $C_{a}$ for the cyclic group of order $a$. Consider the natural projection $C_{4}\to\gal(\CC/\RR)=C_{2}$ giving a non-faithful action of $C_{4}$ on $\CC$ and the Galois equivariant action on $\PP^{2}_{\CC}$ given by $(a:b:c)\mapsto (-\bar{b}:\bar{a}:\bar{c})$. Clearly, $C_{4}$ stabilizes $S$, hence we get a homomorphism $C_{4}\to\cN_{S}$. Lemma~\ref{lemma:aut} implies that $\cN_{S}$ is an extension of $\gal(\CC/\RR)=C_{2}$ by $C_{2}$, hence $C_{4}\to\cN_{S}$ is an isomorphism and $\cG_{S}\simeq [\spec \CC/C_{4}]$. To conclude, it is enough to show that $[\spec \CC/C_{4}](\RR)=\emptyset$.

By definition, an $\RR$-point of $[\spec \CC/C_{4}]$ corresponds to a $C_{4}$-torsor over $\RR$ with a $C_{4}$-equivariant map to $\spec \CC$. There are two $C_{4}$-torsors over $\RR$, the trivial one and $T=\spec \CC\cup\spec \CC$, and neither of them has $C_{4}$-equivariant morphisms to $\spec \CC$. This is clear for the trivial torsor, while for $T$ it follows from the fact that $C_{2}\subset C_{4}$ acts non-trivially on each copy of $\spec \CC\subset T$.

Notice that $\bP_{S}(\RR)=\PP^{2}_{\CC}/C_{4}(\RR)$ is non-empty: however, the only real point $(0:0:1)$ is an $A_{1}$-singularity, and hence we cannot apply the Lang-Nishimura theorem for stacks to conclude that $\cP_{S}(\RR)$ is non-empty (in fact, $\cP_{S}(\RR)=\emptyset$). We mention that, for most types of $2$-dimensional quotient singularities (but not $A_{1}$-singularities), the Lang-Nishimura theorem is still valid, see \cite[\S 6]{giulio-angelo-moduli} and \cite{giulio-tqs2} for details.

\bibliographystyle{amsalpha}
\bibliography{p2}

\end{document}